\documentclass[reqno,a4paper,11pt]{amsart}

\usepackage[a4paper, margin=1in]{geometry}
\usepackage[utf8]{inputenc}
\usepackage[T1]{fontenc}
\usepackage[english]{babel}
\usepackage[babel]{microtype}
\usepackage[dvipsnames]{xcolor}
\usepackage{amssymb}
\usepackage{mathtools}
\usepackage{commath}
\usepackage{amsthm}
\usepackage{graphicx}

\theoremstyle{plain}
\newtheorem{theorem}{Theorem}[section]
\newtheorem{lemma}[theorem]{Lemma}
\newtheorem{proposition}[theorem]{Proposition}
\newtheorem{question}[theorem]{Question}
\newtheorem{problem}[theorem]{Problem}
\theoremstyle{definition} 
\newtheorem{remark}[theorem]{Remark}

\usepackage[pdfborderstyle={/S/U/W 0},unicode]{hyperref}
\hypersetup{
    colorlinks=true,
    linkcolor=blue!85!black,
    citecolor=blue!85!black,
    urlcolor=blue!85!black,
}

\numberwithin{equation}{section}

\usepackage{silence}




\renewcommand{\le}{\leqslant}
\renewcommand{\ge}{\geqslant}
\renewcommand{\d}{\mathrm{d}}
\newcommand{\ve}{\varepsilon}
\newcommand{\brax}[1]{\left(#1\right)}
\newcommand{\intdist}[1]{\left\|#1\right\|}
\newcommand{\Case}[1]{\smallskip \noindent \textit{Case #1.}}

\begin{document}

\title[Point sets avoiding near-integer distances]{Point sets avoiding near-integer distances}

\author[R. Goenka]{Ritesh Goenka}
\address{Mathematical Institute, University of Oxford, Oxford, United Kingdom}
\email{ritesh.goenka@maths.ox.ac.uk}

\author[K. Moore]{Kenneth Moore}
\address{HUN-REN Alfr\'ed R\'enyi Institute of Mathematics, Budapest, Hungary}
\email{kjmoore@renyi.hu}

\begin{abstract}
    Let $d \in \mathbb{N}$, $\delta \in (0, 1/2)$, and $X > 0$. Denote by $N_d(X, \delta)$ the maximum number of points in a subset of the closed Euclidean ball of radius $X$ in $\mathbb{R}^d$ such that every pairwise distance is at least $\delta$ away from any integer. In the planar case, S\'ark\"ozy proved that for every $\varepsilon > 0$, $N_2(X, \delta) = \Omega_\delta(X^{1/2-\varepsilon})$ as $X \rightarrow \infty$ whenever $\delta$ is sufficiently small in terms of $\varepsilon$, while Konyagin proved the almost matching upper bound $N_2(X,\delta) = O_\delta(X^{1/2})$.
    
    We study this problem in higher dimensions, addressing a question of Erd\H{o}s and S\'ark\"ozy. Extending S\'ark\"ozy's construction, we show that for every $\varepsilon > 0$, $N_3(X, \delta) = \Omega_\delta(X^{1-\varepsilon})$ for $\delta$ sufficiently small in terms of $\varepsilon$. We also provide a lifting lemma from integer distance sets to sets avoiding near-integer distances via bilipschitz embeddings of snowflaked Euclidean spaces. This allows us to prove a linear lower bound $N_4(X,\delta) = \Omega_\delta(X)$ for all sufficiently small $\delta$. Finally, adapting Konyagin's approach, we prove the upper bound $N_d(X, \delta) = O_{d, \delta}(X^{d/2})$ for all $d \in \mathbb{N}$.
\end{abstract}


\maketitle

\section{Introduction}
\label{sec:intro}

For $d \in \mathbb{N}$, $\delta \in (0, 1/2)$, and $X > 0$, let $N_d(X, \delta)$ denote the largest integer $n$ such that there exist points $p_1, p_2, \dots, p_n \in B_d(0, X)$ satisfying
\begin{equation}
\label{eqn:main}
    \intdist{\, |p_i - p_j| \,} \ge \delta, \quad \text{for all } i \ne j,
\end{equation}
where $B_d(0, X)$ denotes the closed Euclidean ball of radius $X$ in $\mathbb{R}^d$, $| \cdot |$ denotes the Euclidean norm, and $\intdist{t}$ denotes the distance from $t$ to the nearest integer. Note that $N_d(X, \delta)$ is finite and well defined since the point set $P = \set{p_i \,\colon\, i \in [n]}$ must be $\delta$-separated. Throughout the paper, we use asymptotic notation with respect to $X \rightarrow \infty$.

Erd\H{o}s (see~\cite{Erd,EG}) conjectured that $N_2(X, \delta) = \omega(1)$ and $N_2(X, \delta) = o(X)$ as $X \rightarrow \infty$ while keeping $\delta$ constant. The first conjecture was proved by Graham (see \cite{Sar2}) who showed that $N_2(X, \delta) = \Omega(\log X)$ for $\delta \le 1/10$ and the second was proved by S\'ark\"ozy~\cite{Sar1} who showed that
\begin{equation}
\label{eqn:SarUpper}
    N_2(X, \delta) = O_\delta(X/\log \log X).
\end{equation}
Both of these bounds were subsequently improved: S\'ark\"ozy~\cite{Sar2} showed that for every $\ve > 0$, $N_2(X, \delta) = \Omega_\delta(X^{1/2-\ve})$ for $\delta$ small enough in terms of $\ve$ and Konyagin~\cite{Kon} showed the almost matching upper bound $N_2(X, \delta) = O_\delta(X^{1/2})$. 

In contrast to the planar case, the literature on this problem in higher dimensions is sparse. Erd\H{o}s and S\'ark\"ozy~\cite{ES} asked the following question about the growth of $N_d(X, \delta)$.

\begin{question}
\label{ques:main}
    How rapidly does $N_d(X, \delta)$ increase for $d \rightarrow \infty$? Does there exist a positive integer $d$ such that
    \begin{equation*}
        \lim_{X \rightarrow \infty} \frac{N_d(X, \delta)}{X} = \infty
    \end{equation*}
    for some $\delta > 0$? In view of \eqref{eqn:SarUpper}, this would imply
    \begin{equation*}
        \lim_{X \rightarrow \infty} \frac{N_d(X, \delta)}{N_2(X, \delta)} = \infty.
    \end{equation*}
\end{question}

We first extend Konyagin's Fourier-analytic approach to higher dimensions. This gives a general upper bound for $N_d(X, \delta)$ in every dimension. The precise exponent depends on the congruence class of $d$ modulo $4$, which arises from the inapplicability of our Fourier-analytic argument in dimensions $d \equiv 3 \pmod{4}$.

\begin{theorem}
\label{thm:upper}
    For every $d \in \mathbb{N}$ and $\delta \in (0, 1/2)$, we have $N_d(X, \delta) = O_{d,\delta}(f_d(X))$, where
    \begin{equation*}
        f_d(X) \coloneqq \begin{cases}
            X^{\frac{d}{2}}, &\text{ if } d \equiv 3 \pmod{4},\\
            X^{\frac{d-1}{2}} \log X, &\text{ if } d \equiv 0 \pmod{4},\\
            X^{\frac{d-1}{2}}, &\text{ otherwise}.
        \end{cases}
    \end{equation*}
\end{theorem}

The main technical novelty in the proof of Theorem~\ref{thm:upper} is a trigonometric lemma (Lemma~\ref{lem:dual_lemma}) which states that there exist trigonometric polynomials of a certain form that are uniformly negative over the interval $[2\pi \delta, 2\pi (1-\delta)]$. This generalises a similar trigonometric lemma from Konyagin's paper~\cite[Lemma~1]{Kon}. However, the proof of Lemma~\ref{lem:dual_lemma} is new and uses several analytic tools such as Cauchy transforms, the Vivanti--Pringsheim theorem, the Hahn--Banach separation theorem, and the Riesz representation theorem for measures.

Observe that Theorem~\ref{thm:upper} implies that for each fixed $d$, the growth of $N_d(X, \delta)$ is at most polynomial of degree at most $d/2$. To complement this upper bound, we give two lower bound constructions. The first is a three-dimensional adaptation of S\'ark\"ozy's construction. The second is based on integer distance sets and snowflake embeddings. If $A \subset B_k(0, R)$ has all pairwise Euclidean distances integers and $\phi \,\colon\, (B_k(0, R),\|\cdot\|_2^{1/2}) \rightarrow \mathbb{R}^m$ is bilipschitz, then for a suitable choice of the scaling parameter $\lambda$, the lifted set
\begin{equation*}
    \set{(p,\lambda \phi(p)) \,\colon\, p\in A} \subset \mathbb{R}^{k+m}
\end{equation*}
avoids a fixed neighbourhood of the integers. This relies on the fact that the distance between two lifted points has the form $(a^2+r)^{1/2}$, where $a\in\mathbb{N}$ and $r$ is comparable to $a$. Lemma~\ref{lem:sarkozy} then places this distance in the interval $(a+\delta,a+1-\delta)$. Applying this principle to $A = \set{1,\dots, \lfloor X/2 \rfloor}$ along with a bilipschitz embedding of $(B_1(0, \lfloor X/2 \rfloor),|\cdot|^{1/2})$ into $\mathbb{R}^3$, obtained using Assouad's theorem~\cite{Ass}, gives a linear construction in dimension four.

\begin{theorem}
\label{thm:lower}
    Let $\ve \in (0, 1)$. Then there exists $\delta_0 = \delta_0(\ve) > 0$ such that for every $0 < \delta \le \delta_0$,
    \begin{equation*}
        N_3(X, \delta) = \Omega_\delta(X^{1-\ve})
        \quad \text{and} \quad
        N_4(X, \delta) = \Omega_\delta(X).
    \end{equation*}
\end{theorem}

Theorem~\ref{thm:lower} gives an almost linear lower bound in dimension $3$. Further, since $N_d(X, \delta)$ is nondecreasing in $d$, it gives a linear lower bound in every dimension $d \ge 4$. In particular, together with Konyagin's square root upper bound in two dimensions, it shows that the asymptotic behaviour of $N_d(X, \delta)$ in higher dimensions diverges substantially from the planar case, namely, $N_d(X, \delta) / N_2(X, \delta) \rightarrow \infty$ for $d \ge 3$ and sufficiently small fixed $\delta$. The stronger superlinear lower bound in Question~\ref{ques:main} remains open.

It is also natural to view this problem in the broader context of distance phenomena. A closely related extremal question concerns the maximum size of integer distance sets, which, as mentioned above, can be used to construct sets avoiding near integer distances via suitable embeddings. There is also a measurable analogue of integer distance avoiding sets due to Erd\H{o}s, Graham, and S\'ark\"ozy (see~\cite{Erd2}). See Section~\ref{sec:conclusion} for an extended discussion about these as well as other related problems.

From a different perspective, it is also related to the works of Furstenberg, Katznelson, and Weiss~\cite{FKW} and Falconer~\cite{Fal}, which study the structure of distance sets in dense subsets of $\mathbb{R}^d$. See also the work of Bukh~\cite{Buk} where the problem of determining the maximum density of a set with a finite set of excluded distances is studied.

The rest of this paper is organised as follows. In Section~\ref{sec:upper}, we prove the upper bound in Theorem~\ref{thm:upper}. In Section~\ref{sec:lower}, we prove the lower bounds in Theorem~\ref{thm:lower}, including both the generalisation of S\'ark\"ozy's construction and the snowflake-lift construction. Finally, we conclude with a discussion about related questions and several directions for future research in Section~\ref{sec:conclusion}.

\section*{Acknowledgements}

The authors thank Peter Keevash and J\'anos Pach for helpful conversations. We also thank Thomas Bloom for maintaining the collection of Erd\H{o}s problems~\cite[see problems~465, 466, and 953]{Blo} where the authors first came across this problem. We acknowledge the use of Gemini 3.1 Pro (Google) and ChatGPT5 (OpenAI) in generating some proof ideas. RG is supported by a joint Clarendon Fund and Exeter College SKP scholarship. KM is supported by ERC Advanced Grants ``GeoScape'', no. 882971 and ``ERMiD'', no. 101054936.

\section{Upper bound}
\label{sec:upper}

We begin with two lemmas that give elementary bounds on $N_d(X, \delta)$. We then state the main ingredient for the upper bound, Lemma~\ref{lem:Kon}, a Fourier-analytic lemma inspired by Konyagin's method~\cite{Kon}. We first show how an inductive argument using this lemma implies Theorem~\ref{thm:upper}. The remainder of the section is devoted to the proof of Lemma~\ref{lem:Kon}, whose primary auxiliary input is the trigonometric lemma, Lemma~\ref{lem:dual_lemma}. In order to prove this trigonometric lemma, we first establish its dual in Lemma~\ref{lem:integeral_of_trig}.

\begin{lemma}
\label{lem:oned}
    Let $\delta \in (0, 1/2)$ and $X > 0$. Then $N_1(X, \delta) \le 1/\delta$.
\end{lemma}

\begin{proof}
    Fix $\delta \in (0, 1/2)$. Let $Q$ be any set of points in $\mathbb{R}$ satisfying \eqref{eqn:main}. Let $\widetilde{Q}$ be the set consisting of the projections of points in $Q$ onto the torus $\mathbb{R} / \mathbb{Z}$. Then \eqref{eqn:main} implies that no two elements of $Q$ map to the same element in $\widetilde{Q}$ under this projection and $\widetilde{Q}$ is $\delta$-separated, which together imply $|Q| = |\widetilde{Q}| \le 1/\delta$. Therefore, $N_1(X, \delta) \le 1/\delta$ for all $X > 0$.
\end{proof}

\begin{lemma}
\label{lem:ind}
    Let $d \in \mathbb{N}$, $\delta \in (0, 1/2)$, and $X > 0$. Then $N_{d+1}(X, \delta) \le \lceil 4X/\delta \rceil N_d(X, \delta/2)$.
\end{lemma}

\begin{proof}
    Let $P$ be any subset of points in $B_{d+1}(0, X)$ satisfying \eqref{eqn:main}, and $\mathbf{e}_1 = (1,0, \dots, 0) \in \mathbb{R}^{d+1}$. Divide the ball into slabs of width $\delta/2$ perpendicular to the direction specified by $\mathbf{e}_1$. Consider the subset $S \subseteq P$ of points in any one of these slabs. Projecting this set onto the hyperplane $x_1 = 0$ gives us a subset $\widetilde{S}$ of $B_d(0, X)$. Let $p_1$ and $p_2$ be any two distinct points in $S$. Then their projections $\widetilde{p}_1$ and $\widetilde{p}_2$ satisfy
    \begin{equation*}
        \left|\, |\widetilde{p}_1 - \widetilde{p}_2| - |p_1 - p_2|\, \right| \le \delta/2,
    \end{equation*}
    which combined with \eqref{eqn:main} imply
    \begin{equation*}
        \intdist{\, |\widetilde{p}_1 - \widetilde{p}_2| \,} \ge \delta/2.
    \end{equation*}
    Therefore, the projection is injective on $S$ and $|S| = |\widetilde{S}| \le N_d(X, \delta/2)$. Since the total number of slabs is bounded above by $\lceil 4X/\delta \rceil$, this yields $|P| \le \lceil 4X/\delta \rceil N_d(X, \delta/2)$.
\end{proof}

Starting with the bound in Lemma~\ref{lem:oned} and repeatedly applying Lemma~\ref{lem:ind} yields the bound $N_d(X, \delta) = O_{d,\delta}(X^{d-1})$ for all $d \in \mathbb{N}$ and $\delta \in (0, 1/2)$. The following lemma will allow us to improve upon this bound.

\begin{lemma}
\label{lem:Kon}
    Let $d \in \mathbb{N}$ satisfy $d \not\equiv 2 \pmod 4$. Suppose $N_d(X, \delta) = O_{d,\delta}(f_d(X))$ for all $\delta \in (0, 1/2)$, for some locally integrable function $f_d \,\colon\, (1,\infty) \rightarrow \mathbb{R}_{> 0}$. Then
    \begin{equation*}
        N_{d+1}(X, \delta) = O_{d,\delta}\left(X^{d/2}\left(1 + \int_{2}^{4X} \frac{f_d(t)}{t^{d/2+1}} \d t\right) + f_d(4X)\right) \quad \text{ for all } \delta \in (0, 1/2).
    \end{equation*}
\end{lemma}

We now prove Theorem~\ref{thm:upper} assuming the above lemma.

\begin{proof}[Proof of Theorem~\ref{thm:upper}]
    The proof proceeds via an induction on the dimension $d$. For the base case $d = 1$, Lemma~\ref{lem:oned} implies $N_1(X, \delta) = O_\delta(1)$. This verifies the result for $d = 1$.

    For the inductive step, suppose that the result is true for dimension $d$. We will show that the result is also true for dimension $d+1$. We divide our analysis into four cases.

    \Case{1} $d \equiv 1 \pmod 4$. In this case, we shall use Lemma~\ref{lem:Kon}. By the inductive hypothesis, we can take $f(t) = t^{(d-1)/2}$, which implies
    \begin{equation*}
        \int_{2}^{4X} \frac{f(t)}{t^{d/2+1}} \d t = \int_{2}^{4X} t^{-3/2} \d t = O_\delta(1).
    \end{equation*}
    Using this in Lemma~\ref{lem:Kon} yields $N_{d+1}(X, \delta) = O_{d,\delta}(X^{d/2})$ for all $\delta \in (0, 1/2)$.

    \Case{2} $d \equiv 2 \pmod 4$. In this case, we shall use Lemma~\ref{lem:ind}. By the inductive hypothesis, we have $N_d(X, \delta) = O_{d,\delta}(X^{(d-1)/2})$ for all $\delta \in (0, 1/2)$. Using this in Lemma~\ref{lem:ind} yields $N_{d+1}(X, \delta) = O_{d,\delta}(X^{(d+1)/2})$ for all $\delta \in (0, 1/2)$.

    \Case{3} $d \equiv 3 \pmod 4$. In this case, we shall use Lemma~\ref{lem:Kon}. By the inductive hypothesis, we can take $f(t) = t^{d/2}$, which implies
    \begin{equation*}
        \int_{2}^{4X} \frac{f(t)}{t^{d/2+1}} \d t = \int_{2}^{4X} t^{-1} \d t = O_\delta(\log X).
    \end{equation*}
    Using this in Lemma~\ref{lem:Kon} yields $N_{d+1}(X, \delta) = O_{d,\delta}(X^{d/2} \log X)$ for all $\delta \in (0, 1/2)$.

    \Case{4} $d \equiv 0 \pmod 4$. In this case, we shall use Lemma~\ref{lem:Kon}. By the inductive hypothesis, we can take $f(t) = t^{(d-1)/2} \log t$, which implies
    \begin{equation*}
        \int_{2}^{4X} \frac{f(t)}{t^{d/2+1}} \d t = \int_{2}^{4X} \frac{\log t}{t^{3/2}} \d t = O_\delta(1).
    \end{equation*}
    Using this in Lemma~\ref{lem:Kon} yields $N_{d+1}(X, \delta) = O_{d,\delta}(X^{d/2})$ for all $\delta \in (0, 1/2)$.

    The four cases above together finish the proof of the inductive step, and therefore the theorem follows by induction.
\end{proof}

The remainder of this section is devoted to proving Lemma~\ref{lem:Kon}. For $\delta \in (0, 1/2)$, we define
\begin{equation*}
    I_\delta \coloneqq [2\pi\delta, 2\pi(1-\delta)].
\end{equation*}
Suppose that the distance between two points $p$ and $q$ in $\mathbb{R}^d$ satisfies $\intdist{|p - q|} \ge \delta$. Then $2 \pi |p-q|$ modulo $2 \pi$ lies in the interval $I_\delta$. Our goal in the following two lemmas is to establish the existence of trigonometric polynomials of a certain form that are uniformly negative on the interval $I_\delta$, and hence also on the set $\set{2 \pi r \,\colon\, r \in \mathbb{R}_{\ge 0}, \intdist{r} \ge \delta}$.

\begin{lemma}
\label{lem:integeral_of_trig}
    Let $\delta \in (0, \frac{1}{2})$, and let $\ell \in \mathbb{N}$. Then a nonzero finite positive Borel measure $\mu$ supported on the interval $I_\delta$ satisfying
    \begin{equation}
    \label{eq:negative_integral_of_trig}
        \int_{I_\delta} \cos\left(kx-\frac{\ell\pi}{4}\right) \d \mu \ge 0
        \quad \text{ for all } k \in \mathbb{N},
    \end{equation}
    exists if and only if $\ell \equiv 2 \pmod{4}$.
\end{lemma}

\begin{proof}
    First, suppose $\ell \equiv 2 \pmod{4}$. Then $\cos\left(kx-\frac{\ell\pi}{4}\right)$ is odd about $x = \pi$ for each $k \in \mathbb{N}$, which implies 
    \begin{equation*}
        \int_{2\pi\delta}^{2\pi(1-\delta)} \cos\left(kx-\frac{\ell\pi}{4}\right) \d x = 0
        \quad \text{ for all } k \in \mathbb{N}.
    \end{equation*}
    Therefore, \eqref{eq:negative_integral_of_trig} is satisfied by taking $\mu$ to be the Lebesgue measure restricted to $I_\delta$.
    
    Conversely, suppose $\ell \not \equiv 2 \pmod{4}$. Suppose for contradiction that there exists a nonzero finite positive Borel measure $\mu$ satisfying \eqref{eq:negative_integral_of_trig}. We may assume without loss of generality that $\mu$ is a probability measure since it still satisfies \eqref{eq:negative_integral_of_trig} after rescaling by a positive constant. Let $A = A(\ell) \coloneqq \cos(\ell\pi/4)$ and $B = B(\ell) \coloneqq \sin(\ell\pi/4)$. Note that $A \ne 0$ and
    \begin{equation*}
        \cos\left(\theta - \frac{\ell\pi}{4}\right) = A \cos \theta + B \sin \theta
    \end{equation*}
    for all $\theta \in \mathbb{R}$. For $k \in \mathbb{N}$, let $\alpha_k = \int_{I_\delta} (A \cos(kx) + B \sin(kx)) \d \mu$.  By hypothesis, $\alpha_k \ge 0$. Further, since $\mu$ is a probability measure, we have $|\alpha_k| \le |A| + |B|$. Consider the power series
    \begin{equation*}
        F(z) \coloneqq \sum_{k=1}^\infty \alpha_k z^k.
    \end{equation*}
    Since $|\alpha_k|$ is uniformly bounded above, the radius of convergence is $R \ge 1$. For any $z$ in the open unit disc, the power series converges absolutely. By the dominated convergence theorem, we get
    \begin{equation}
    \label{eqn:Fdef}
        F(z) = \int \sum_{k=1}^\infty z^k (A \cos(kx) + B \sin(kx)) \d \mu.
    \end{equation}
    Now, we have the standard geometric series evaluations
    \begin{equation*}
        \sum_{k=1}^\infty z^k \cos(kx) = \frac{z \cos x - z^2}{1 - 2z \cos x + z^2}, \text{ and } \sum_{k=1}^\infty z^k \sin(kx) = \frac{z \sin x}{1 - 2z \cos x + z^2},
    \end{equation*}
    which can be obtained by looking at the real and imaginary parts of the infinite geometric series $\sum_{k=1}^\infty (z e^{ix})^k$. Plugging these into \eqref{eqn:Fdef}, we obtain
    \begin{equation*}
        F(z) = \int K(z, x) \d \mu,
    \end{equation*}
    where
    \begin{equation*}
        K(z, x) = \frac{A(z \cos x - z^2) + B z \sin x}{1 - 2z \cos x + z^2}.
    \end{equation*}
    Writing the expression for $K(z, x)$ in a partial fractions form, we obtain
    \begin{equation*}
        K(z, x) = -A - \frac{A+iB}{2} \frac{e^{ix}}{z - e^{ix}} - \frac{A-iB}{2} \frac{e^{-ix}}{z - e^{-ix}}. 
    \end{equation*}
    Let us define another measure $\widetilde{\mu}$ supported on $I_\delta$ by $\widetilde{\mu}(E) = \mu(2\pi - E)$. Then we have
    \begin{equation}
    \label{eqn:part-frac}
        F(z) = -A - \int \frac{e^{ix}}{z - e^{ix}} \left(\frac{A + iB}{2} \d \mu + \frac{A - iB}{2} \d \widetilde{\mu}\right).
    \end{equation}
    Let $w = e^{ix}$. Define a complex measure $\nu$ on the closed arc $C_\delta \coloneqq \set{e^{ix} \,\colon\, x \in I_\delta}$ as
    \begin{equation}
    \label{eqn:nudef}
        \d \nu(w) = \left(\frac{A + iB}{2} \d \mu(x) + \frac{A - iB}{2} \d \widetilde{\mu}(x)\right) w.
    \end{equation}
    Then we can rewrite \eqref{eqn:part-frac} as
    \begin{equation}
    \label{eqn:FG}
        F(z) = - A + G(z),
    \end{equation}
    where
    \begin{equation*}
        G(z) = \int_{\mathbb{T}} \frac{1}{w-z} \d \nu(w).
    \end{equation*}
    Note that the integral on the right-hand side above is the Cauchy transform of the finite complex measure $\nu$ (see \cite[Section~1.5]{Tol} for a definition). It follows from \cite[Proposition~1.13]{Tol} that $G$ is locally integrable with respect to the Lebesgue measure on $\mathbb{C}$ and is analytic on $\mathbb{C} \setminus \text{Supp}(\nu)$, where $\text{Supp}(\nu)$ denotes the support of measure $\nu$. In particular, this implies that $G$ is analytic at $z = 1$, which further implies that $F$ extends analytically to $z = 1$. Recall that the radius of convergence $R$ of the power series defining $F$ satisfies $R \ge 1$. Suppose $R = 1$. Then, by the Vivanti--Pringsheim theorem (see \cite[Theorem~IV.6]{FS}), $z = 1$ must be a singular point of $F$. But we just proved that $F$ extends analytically to $z = 1$. Therefore, $R > 1$.

    Now note that $G(z)$ is an analytic function on the domain $\mathbb{C} \setminus C_\delta$ while $F(z) + A$ is an analytic function on the domain $\Delta(0, R)$, where $\Delta(0, r)$ denotes the open disc of radius $r$ centered at the origin. Note that 
    \begin{equation*}
        (\mathbb{C} \setminus C_\delta) \cup \Delta(0, R) = \mathbb{C}
    \end{equation*}
    since $R > 1$. Moreover, by \eqref{eqn:FG}, we have $F(z) + A = G(z)$ for all $z \in \Delta(0, 1)$. Now since $\Delta(0, 1)$ is a domain contained in the intersection of $(\mathbb{C} \setminus C_\delta)$ and $\Delta(0, R)$, it follows from the identity theorem that there exists an analytic continuation $\widetilde{G}$ of $G$ that is entire. Further, since $C_\delta$ is a set of Lebesgue measure zero in $\mathbb{C}$, $G$ and $\widetilde{G}$ induce the same distribution, say $T$, on $\mathbb{C}$. Therefore, it follows from \cite[Theorem~1.14]{Tol} that
    \begin{equation*}
        \overline{\partial} (T) = -\pi T_\nu,
    \end{equation*}
    where $\overline{\partial}$ is the Cauchy--Riemann operator and $T_\nu$ is the complex distribution induced by $\nu$. But $\overline{\partial}(T) = 0$ since $\widetilde{G}$ is entire. This implies $T_\nu = 0$. Now since $C_c^\infty(\mathbb{C})$ is uniformly dense in $C_c(\mathbb{C})$ and every function in $C(C_\delta)$ can be extended to a function in $C_c(\mathbb{C})$, the condition $T_\nu = 0$ implies $\int_{C_\delta} f \d \nu = 0$ for all $f \in C(C_\delta)$. By the uniqueness part of the Riesz representation theorem for complex Borel measures (see \cite[Theorem~7.17]{Fol}), it follows that $\nu = 0$. Using this in \eqref{eqn:nudef}, and noting that $0 \not\in C_\delta$, we obtain the equality of complex measures
    \begin{equation*}
        \frac{A}{2} (\d \mu + \d \widetilde{\mu}) + \frac{iB}{2} (\d \mu - \d \widetilde{\mu}) = 0.
    \end{equation*}
    Taking real part on both sides gives
    \begin{equation*}
        \frac{A}{2} (\d \mu + \d \widetilde{\mu}) = 0.
    \end{equation*}
    Evaluating this measure on $I_\delta$, we obtain
    \begin{equation*}
        \frac{A}{2}(\mu(I_\delta) + \widetilde{\mu}(I_\delta)) = 0.
    \end{equation*}
    Finally, since $\mu(I_\delta) = \widetilde{\mu}(I_\delta) = 1$ and $A \ne 0$, this is a contradiction.
\end{proof}

Next we prove the dual form of Lemma~\ref{lem:integeral_of_trig}.
\begin{lemma}
    \label{lem:dual_lemma}
    For any $\delta \in (0, \frac{1}{2})$ and $\ell \in \mathbb{N}$ with $\ell \not \equiv 2 \pmod{4}$, there exists a trigonometric polynomial 
    \begin{equation}
        T(x)=\sum_{k=1}^mc_k \cos\brax{kx - \frac{\ell\pi}{4}}
    \end{equation}
    with nonnegative coefficients $(c_k)_{k \in [m]}$ which satisfies 
    \begin{equation}
        A \coloneqq -\max_{x\in I_\delta} T(x) > 0.
    \end{equation}
\end{lemma}

\begin{proof}
    Consider the Banach space $C(I_\delta)$. Let $\mathcal{U}$ denote the cone of strictly negative functions in $C(I_\delta)$, namely,
    \begin{equation*}
        \mathcal{U} \coloneqq \set{f\in C(I_\delta) \,\colon\, f(x) < 0 \text{ for all } x \in  I_\delta}.
    \end{equation*}
    Further, let $\mathcal{C}$ denote the cone of functions generated by finite combinations of $\cos(kx - \frac{\ell\pi}{4})$ terms with nonnegative coefficients, namely,
    \begin{equation*}
        \mathcal{C} \coloneqq \set{ T \in C(I_\delta) \,\colon\, T(x) = \sum_{k=1}^{m} c_k \cos\!\left(kx-\frac{\ell\pi}{4}\right) \text{ for some } m \in \mathbb{N} \text{ and } c_1, \dots, c_m \ge 0 }.
    \end{equation*}
    Note that both $\mathcal{C}$ and $\mathcal{U}$ are nonempty and convex. Moreover, $\mathcal{U}$ is open. Suppose that the statement of the lemma is false. Then $\mathcal{C} \cap \mathcal{U} = \emptyset$. Therefore, the Hahn--Banach separation theorem (see \cite[Theorem~3.4]{Rud}) implies that there exists a nonzero continuous linear functional $\Lambda$ that separates them; that is, there is a real number $\alpha$ and a linear functional $\Lambda \in C(I_\delta)^*$ such that $\Lambda(f) \ge \alpha$ for all $f \in \mathcal{C}$ and $\Lambda(g) < \alpha$ for all $g \in \mathcal{U}$. Since $0 \in \mathcal{C}$, we have $\alpha \le 0$. On the other hand, if $g \in \mathcal{U}$, then $tg \in \mathcal{U}$ for every $t > 0$, and hence $t \Lambda(g) < \alpha$ for every $t > 0$. Letting $t \rightarrow 0^+$ gives $\alpha \ge 0$. Thus $\alpha = 0$.
    
    We now claim that $\Lambda$ is a positive functional. Indeed, if $f \in C(I_\delta)$ and $f \ge 0$, then $- f - \ve \in \mathcal{U}$ for every $\ve > 0$. Hence, we have
    \begin{equation*}
        - \Lambda(f) - \ve \Lambda(1) < 0.
    \end{equation*}
    Letting $\ve \rightarrow 0^+$ gives $\Lambda(f) \ge 0$. Therefore, by the Riesz--Markov--Kakutani representation theorem (see \cite[Theorem~2.14]{Rud2}), there exists a unique finite positive Borel measure $\mu$ supported on $I_\delta$ such that $\Lambda(f) = \int_{I_\delta} f(x) \d \mu(x)$ for $f \in C(I_\delta)$. Note that $\mu$ is not the zero measure since $\Lambda$ is a nonzero functional. Further, since $\Lambda(f) \ge 0$ for all $f \in \mathcal{C}$, we conclude that
    \begin{equation*}
        \int_{I_\delta} \cos \left(kx-\frac{\ell\pi}{4}\right) \d \mu(x) \ge 0
    \end{equation*}
    for all $k \in \mathbb{N}$. However, the existence of $\mu$ contradicts Lemma~\ref{lem:integeral_of_trig}.
\end{proof}

With this technical lemma proven, we are ready to take the final steps in the proof.

\begin{proof}[Proof of Lemma~\ref{lem:Kon}] 
    Let $\set{p_1, \dots, p_n} \subset \mathbb{R}^{d+1}$ be a set of $n$ points in the ball of radius $X$, such that the distance between each pair of distinct points is at least $\delta$ away from every integer. Let $\d \sigma$ denote the normalised surface measure on $S^d$. The argument begins with the following formula for the Fourier transform of this measure:
    \begin{equation}
    \label{eq:spherical-integral}
        \int_{S^d} e^{2\pi i k \langle y, \omega \rangle} \, \d \sigma(\omega) = C_d\, (k|y|)^{-\nu} J_{\nu}(2\pi k|y|),
    \end{equation}
    where $\nu = \frac{d-1}{2}$, $C_d > 0$ is a constant depending only on $d$, and $J_\nu$ denotes the Bessel function of the first kind.

    For each integer $k \ge 1$ and direction $\omega \in S^d$, define
    \begin{equation*}
        A_k(\omega) = \sum_{j=1}^n e^{2\pi i k \langle p_j, \omega \rangle}.
    \end{equation*}
    Then we have
    \begin{equation*}
        |A_k(\omega)|^2 = \sum_{i,j} e^{2\pi i k \langle p_i - p_j, \omega \rangle}.
    \end{equation*}
    Integrating over $\omega$ on the sphere $S^{d}$ and using \eqref{eq:spherical-integral} yields
    \begin{equation}
    \label{eq:integral-sum}
        \int_{S^d} |A_k(\omega)|^2 \, \d \sigma(\omega) = n + C_d k^{-\nu} \sum_{i \ne j} r_{ij}^{-\nu} J_\nu(2\pi k r_{ij}),
    \end{equation}
    where $r_{ij} = |p_i - p_j|$. Let $m \in \mathbb{N}$ and $(b_k)_{k \in [m]}$ be a sequence of nonnegative weights, to be chosen later. Summing \eqref{eq:integral-sum} over $k \in [m]$ with weights $b_k$, we obtain
    \begin{equation}
    \label{eq:sum-sphere}
        \sum_{k=1}^m b_k \int_{S^d} |A_k(\omega)|^2 \, \d \sigma(\omega) = n \sum_{k=1}^m b_k + C_d \sum_{i \ne j} \sum_{k=1}^m b_k k^{-\nu} r_{ij}^{-\nu} J_\nu(2\pi k r_{ij}).
    \end{equation}
    The left-hand side of \eqref{eq:sum-sphere} is nonnegative, so the right-hand side is nonnegative as well. The Bessel function satisfies the asymptotic formula
    \begin{equation*}
        J_\nu(u) = \sqrt{\frac{2}{\pi u}} \left[\cos\left(u - \frac{(2\nu+1)\pi}{4}\right) + O(u^{-1})\right].
    \end{equation*}
    Note that $2\nu + 1 = d \not \equiv 2 \pmod{4}$. Therefore, by Lemma~\ref{lem:dual_lemma}, there exists a trigonometric polynomial
    \begin{equation*}
        T(x) \coloneqq \sum_{k = 1}^m c_k \cos\left(kx - \frac{(2\nu+1)\pi}{4}\right)
    \end{equation*}
    such that $\max_{x \in I_\delta} T(x) = -A$ for some $A > 0$. Since $T$ is $(2 \pi)$-periodic and $2 \pi r_{ij}$ modulo $2 \pi$ lies in $I_\delta$, we have
    \begin{equation*}
        T(2\pi r_{ij})\le -A \quad \text{ for all } i \ne j.
    \end{equation*}
    Setting $b_k = \frac{\pi k^{\nu+1/2}}{C_d} c_k$ for $k \in [m]$ in \eqref{eq:sum-sphere} and using the Bessel function asymptotic, we obtain
    \begin{equation}
    \label{eq:nearlybounded}
        \begin{split}
            0 &\le n \sum_{k=1}^m b_k + \sum_{i \ne j} r_{ij}^{-\nu-\frac{1}{2}} \sum_{k=1}^m c_k \cos\left(2\pi k r_{ij} - \frac{(2\nu+1)\pi}{4}\right) + O_{d,\delta}\brax{\sum_{i \ne j}r_{ij}^{-\nu-\frac{3}{2}}}
            \\
            &\le O_{d,\delta}(n) - n(n-1) A (2X)^{-\nu-\frac{1}{2}}+ O_{d,\delta}\brax{\sum_{i \ne j}r_{ij}^{-\nu-\frac{3}{2}}}.
        \end{split}
    \end{equation}

    We will now estimate the third term in the right-hand side above carefully. For each fixed $i \in [n]$, denote the number of points within distance $u$ of $p_i$ by
    \begin{equation*}
        N_i(u) \coloneqq |\set{j \neq i \,\colon\, |p_i-p_j| \le u}|.
    \end{equation*}
    Then for each $i \in [n]$, we have 
    \begin{equation}
    \label{eqn:r}
        \sum_{j \in [n] \setminus \set{i}} r_{ij}^{-\nu-\frac{3}{2}} \le \int_{\delta/2}^{4X} u^{-\nu-\frac{3}{2}} \d N_i(u) = (4X)^{-\nu-\frac{3}{2}}N_i(4X) + \left(\nu + \frac{3}{2}\right)\int_{\delta/2}^{4X} N_i(u)u^{-\nu-\frac{5}{2}} \d u.
    \end{equation}
    We break the rightmost integral above into two parts, namely,
    \begin{equation}
    \label{eqn:breakint}
        \int_{\delta/2}^{4X} N_i(u)u^{-\nu-\frac{5}{2}} \d u = \int_{\delta/2}^{2} N_i(u)u^{-\nu-\frac{5}{2}} \d u + \int_{2}^{4X} N_i(u)u^{-\nu-\frac{5}{2}} \d u.
    \end{equation}
    Now note that $N_i(u) \le N_{d+1}(u, \delta)$. By Lemma~\ref{lem:ind}, $N_{d+1}(u, \delta)$ is in turn bounded above by $\lceil 4u/\delta \rceil N_d(u,\delta/2) = O_{d,\delta}(u f_d(u))$ as $u \rightarrow \infty$. Plugging these estimates in \eqref{eqn:breakint}, we obtain
    \begin{equation*}
        \int_{\delta/2}^{4X} N_i(u)u^{-\nu-\frac{5}{2}} \d u \le \int_{\delta/2}^{2} N_{d+1}(2, \delta) u^{-\nu-\frac{5}{2}} \d u + O_{d,\delta}\left(\int_{2}^{4X} f_d(u) u^{-\nu-\frac{3}{2}} \d u\right).
    \end{equation*}
    Using the above along with the estimate $N_i(u) = O_{d,\delta}(u f_d(u))$ in \eqref{eqn:r}, we obtain
    \begin{equation*}
        \sum_{j \in [n] \setminus \set{i}} r_{ij}^{-\nu-\frac{3}{2}} = O_{d,\delta}\left(X^{-\nu-\frac{1}{2}} f_d(4X) + 1 + \int_{2}^{4X} f_d(u)u^{-\nu-\frac{3}{2}} \d u\right).
    \end{equation*}
    Finally, using the above estimate in \eqref{eq:nearlybounded} and replacing $\nu$ with $\frac{d-1}{2}$, we obtain
    \begin{equation*}
        0 \le O_{d,\delta}(n) - n(n-1) A (2X)^{-\frac{d}{2}}+ O_{d,\delta}\brax{nX^{-\frac{d}{2}}f_d(4X) + n\int_{2}^{4X} f_d(u)u^{-\frac{d}{2}-1} \d u}.
    \end{equation*}
    Taking the second term on the right-hand side above to the left-hand side and multiplying both sides by $X^{\frac{d}{2}}/n$, we obtain
    \begin{equation*}
        n =  O_{d,\delta}\brax{X^{\frac{d}{2}}\brax{1+ \int_{2}^{4X} f_d(u)u^{-\frac{d}{2}-1} \d u} + f_d(4X)},
    \end{equation*}
    as desired.
\end{proof}

\section{Lower bounds}
\label{sec:lower}

In this section, we prove Theorem~\ref{thm:lower}. The first construction is inspired by the two-dimensional construction of S\'ark\"ozy~\cite{Sar2}. We begin with an elementary lemma, which we shall use both in the three-dimensional construction and in the snowflake-lift argument for dimension four. This lemma was used by S\'ark\"ozy for his planar construction and also by Graham for an earlier construction. The idea is that if a distance vector has one coordinate $a$ that is an integer, while the squared contribution of the remaining coordinates is comparable to $a$, then the Euclidean distance is forced to lie safely between two consecutive integers. In the three-dimensional construction, the points are concentrated near a paraboloid around the $a$-axis so that every difference vector has this property.

\begin{lemma}
\label{lem:sarkozy}
    Let $0 < \delta < 1/2$ and $a \in \mathbb{N}$. Suppose $r$ is a positive real number satisfying
    \begin{equation*}
        3 \delta \le \frac{r}{a} \le 2(1-\delta).
    \end{equation*}
    Then
    \begin{equation*}
        \intdist{\sqrt{a^2 + r}} > \delta.
    \end{equation*}
\end{lemma}

\begin{proof}
    We have
    \begin{equation*}
        a^2 + 3\delta a \le  a^2 + r \le a^2 + 2a(1-\delta),
    \end{equation*}
    which implies
    \begin{equation*}
        a^2 + 2\delta a + \delta^2 <  a^2 + r < a^2 + 2a(1-\delta) + (1-\delta)^2.
    \end{equation*}
    Taking square roots throughout, we obtain
    \begin{equation*}
        a + \delta < \sqrt{a^2 + r} < a + (1-\delta),
    \end{equation*}
    which implies the desired result.
\end{proof}

\begin{proposition}
\label{prop:lower3}
    Let $0 < \delta \le \frac{1}{48 \cdot 3^5}$. Then there exists a positive constant $X_0(\delta)$ such that
    \begin{equation*}
        N_3(X, \delta) \ge \delta^{6/5} X^{1-6\delta^{1/5}}, \text{ for all } X \ge X_0(\delta).
    \end{equation*}
\end{proposition}

\begin{proof}
    Let $0 < \delta \le \frac{1}{48 \cdot 3^5}$. The first step is to carefully choose parameters $k$ and $t$, and then describe their relationships to each other and to $X$ and $\delta$. Let $k$ be the unique positive integer such that 
    \begin{equation*}
        \frac{1}{48k^5} \ge \delta > \frac{1}{48(k+1)^5}.
    \end{equation*}
    Then it follows that $k \ge 3$. Let $X_0(\delta) = 1/\delta$, and suppose $X \ge X_0(\delta)$. Let $t$ be the unique nonnegative integer such that 
    \begin{equation*}
        16k^{2t+4} \le X < 16k^{2(t+1)+4}.
    \end{equation*}
    Note that $t$ is well defined since $16 k^4 \le 48 k^5 \le 1/\delta = X_0(\delta) \le X$. By rearranging the equation defining $t$ we have
    \begin{equation}
        \label{eq:rearrange1}
        k^{2t} > \frac{X}{16k^6}.
    \end{equation}
    Similarly, we also have
    \begin{equation}
        \label{eq:rearrange2}
        t \le \frac{\log X}{2 \log k}\quad \text{ and }\quad
        k > \frac{\delta^{-1/5}}{3}.
    \end{equation}

    \begin{figure}[!t]
        \centering
        \includegraphics[width=0.8\linewidth]{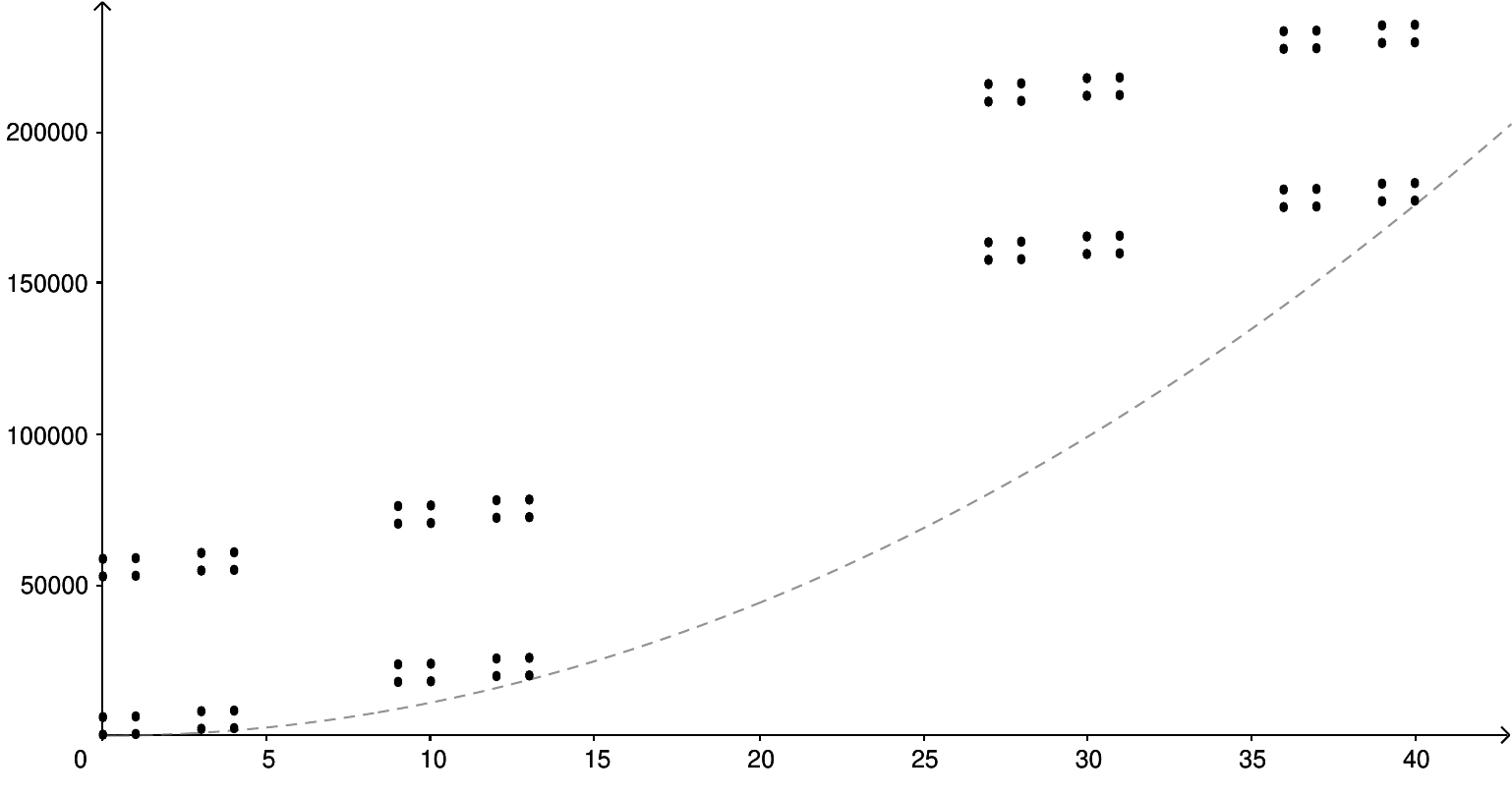}
        \caption{The construction with $k = t = 3$, projected onto the $y$-$z$ plane, together with a section of the paraboloid $z = 110(x^2+y^2)$. The plot is shown at a limited resolution, so some nearby points appear visually clumped.}
        \label{fig:r3}
    \end{figure}

    Now we describe the set $P$ of points in the construction. For each pair of sequences of integers $\{\alpha_i^{(p)}\}_{i=0}^t$, $\{\beta_i^{(p)}\}_{i=0}^t \in \set{0, \dots, k-2}^{t+1}$, define the point $p$ by the coordinates
    \begin{equation*}
        x^{(p)} = \sum_{i=0}^t \alpha_i^{(p)}k^{i}, \quad y^{(p)} = \sum_{i=0}^t \beta_i^{(p)}k^{i}, \quad z^{(p)} = 8 \left(\sum_{i=0}^{t} \alpha_{i}^{(p)} k^{2i+2} + \sum_{i=0}^{t} \beta_{i}^{(p)} k^{2i+3}\right).
    \end{equation*}
    Then for each such point $p = (x^{(p)}, y^{(p)}, z^{(p)})$, we have
    \begin{equation*}
        |x^{(p)}| \le k^{t+1}, \quad |y^{(p)}| \le k^{t+1}, \quad \text{and} \quad |z^{(p)}| \le 8 k^{2t+4}, 
    \end{equation*}
    which implies
    \begin{equation*}
        |p| = \sqrt{(x^{(p)})^2+(y^{(p)})^2+(z^{(p)})^2} \le \sqrt{3(8k^{2t+4})^2} \le 16k^{2t+4} \le X.
    \end{equation*}
    Moreover, using \eqref{eq:rearrange1} and \eqref{eq:rearrange2}, we find the total number of points in our construction is
    \begin{equation*}
        (k-1)^{2(t+1)} \ge (k-1)^{2t} = k^{2t} \left(1 - \frac{1}{k}\right)^{2t} \ge k^{2t} e^{-4t/k} > \frac{X^{1-2/k}}{16 k^6} > \delta^{6/5} X^{1-2/k} > \delta^{6/5} X^{1-6 \delta^{1/5}}.
    \end{equation*}
    For example, see Figure~\ref{fig:r3}, which illustrates the case $X = 10^6$ and $\delta = 5 \times 10^{-5}$, for which $k = t = 3$ and the construction yields $256$ points. Observe that the set has a self-similar structure.

    It remains to verify that integer distances are avoided by every pair of points in $P$. Let $p$ and $q$ be distinct points in $P$ and set
    \begin{equation*}
        a = |z^{(p)} - z^{(q)}|,\ \ b = |x^{(p)} - x^{(q)}|,\ \ \text{and}\ \ c = |y^{(p)} - y^{(q)}|.
    \end{equation*}
    Further, let $i_0$ and $j_0$ be the largest indices where $p$ and $q$ differ in their $x$-coordinate and $y$-coordinate digit sequences, respectively. That is, the largest indices where $\{\alpha_i^{(p)}\}_{i=0}^t$ and $\{\beta_i^{(p)}\}_{i=0}^t$ differ from $\{\alpha_i^{(q)}\}_{i=0}^t$ and $\{\beta_i^{(q)}\}_{i=0}^t$, respectively. We adopt the convention that $i_0 = -1$ if $p$ and $q$ do not differ in their $x$-coordinate sequences, and similarly for $j_0$. Additionally, we set $\alpha_{-1}^{(p)} = \alpha_{-1}^{(q)} = \beta_{-1}^{(p)} = \beta_{-1}^{(q)} = 0$. Further, we interpret sums with upper limit smaller than the lower limit as empty sums. Note that at least one of $i_0$ and $j_0$ must be nonnegative by the distinctness of $p$ and $q$. We have
    \begin{align*}
        \left|(x^{(p)} - x^{(q)}) - (\alpha_{i_0}^{(p)}-\alpha_{i_0}^{(q)})k^{i_0}\right| &\le \sum_{i=0}^{i_0-1} (k-2) k^{i} \le \frac{k-2}{k-1} k^{i_0},\quad \text{and}\\
        \left|(y^{(p)} - y^{(q)}) - (\beta_{j_0}^{(p)}-\beta_{j_0}^{(q)})k^{j_0}\right| &\le \sum_{i=0}^{j_0-1} (k-2) k^{i} \le \frac{k-2}{k-1} k^{j_0},
    \end{align*}
    for $i_0 \ge 0$ and $j_0 \ge 0$, respectively. Combining with the cases $i_0 = -1$ and $j_0 = -1$, we obtain
    \begin{equation}
     \label{eqn:bc}
    \begin{split}
        |\alpha_{i_0}^{(p)}-\alpha_{i_0}^{(q)}| k^{i_0-1} &\le b \le 2 |\alpha_{i_0}^{(p)}-\alpha_{i_0}^{(q)}| k^{i_0},\quad \text{and}\\
        |\beta_{j_0}^{(p)}-\beta_{j_0}^{(q)}| k^{j_0-1} &\le c \le 2 |\beta_{j_0}^{(p)}-\beta_{j_0}^{(q)}| k^{j_0}.
    \end{split}
    \end{equation}
   
    The goal now is to get a similar estimate on $a$, and use these estimates along with Lemma~\ref{lem:sarkozy} to complete the proof. There are two cases, depending on which of $i_0$ or $j_0$ is larger.

    \Case{1} Suppose $i_0 > j_0$. Then we have
    \begin{equation*}
        \frac{1}{8} (z^{(p)} - z^{(q)}) = (\alpha_{i_0}^{(p)}-\alpha_{i_0}^{(q)})k^{2i_0+2} + \sum_{i=0}^{i_0-1} (\alpha_i^{(p)} - \alpha_i^{(q)}) k^{2i+2} + \sum_{i=0}^{j_0} (\beta_i^{(p)} - \beta_i^{(q)}) k^{2i+3},
    \end{equation*}
    which implies
    \begin{equation*}
        \left|\frac{1}{8} (z^{(p)} - z^{(q)}) - (\alpha_{i_0}^{(p)}-\alpha_{i_0}^{(q)})k^{2i_0+2}\right| \le \sum_{i=0}^{i_0-1} (k-2)(k+1) k^{2i+2} \le \frac{k-2}{k-1} \cdot k^{2i_0+2}.
    \end{equation*}
    Rearranging the above inequality yields
    \begin{align*}
        8|\alpha_{i_0}^{(p)}-\alpha_{i_0}^{(q)}| k^{2i_0+1} \le a \le 16 |\alpha_{i_0}^{(p)}-\alpha_{i_0}^{(q)}| k^{2i_0+2}.
    \end{align*}
    The above inequality together with \eqref{eqn:bc} implies
    \begin{align*}
        \frac{b^2 + c^2}{a} &\ge \frac{(|\alpha_{i_0}^{(p)}-\alpha_{i_0}^{(q)}| k^{i_0-1})^2}{16 |\alpha_{i_0}^{(p)}-\alpha_{i_0}^{(q)}| k^{2i_0+2}} \ge \frac{1}{16k^4} \ge 3\delta,\quad \text{and}\\
        \frac{b^2 + c^2}{a} &\le \frac{(2 |\alpha_{i_0}^{(p)}-\alpha_{i_0}^{(q)}| k^{i_0})^2 + (2 |\beta_{j_0}^{(p)}-\beta_{j_0}^{(q)}| k^{j_0})^2}{8|\alpha_{i_0}^{(p)}-\alpha_{i_0}^{(q)}| k^{2i_0+1}} \le \frac{2 (2 |\alpha_{i_0}^{(p)}-\alpha_{i_0}^{(q)}| k^{i_0})^2}{8|\alpha_{i_0}^{(p)}-\alpha_{i_0}^{(q)}| k^{2i_0+1}} \le 1 \le 2(1-\delta). 
    \end{align*}
    Therefore, applying Lemma~\ref{lem:sarkozy} with $r = b^2+c^2$, we obtain $\intdist{\sqrt{a^2+b^2+c^2}} > \delta$ as required.

    \Case{2} Suppose $i_0 \le j_0$. Then we have
    \begin{equation*}
        \frac{1}{8} (z^{(p)} - z^{(q)}) = ((\alpha_{j_0}^{(p)} - \alpha_{j_0}^{(q)}) + k (\beta_{j_0}^{(p)}-\beta_{j_0}^{(q)})) k^{2j_0+2} + \sum_{i=0}^{j_0-1} ((\alpha_i^{(p)} - \alpha_i^{(q)}) + k (\beta_i^{(p)} - \beta_i^{(q)})) k^{2i+2},
    \end{equation*}
    which implies
    \begin{equation*}
        \left|\frac{1}{8} (z^{(p)} - z^{(q)}) - ((\alpha_{j_0}^{(p)} - \alpha_{j_0}^{(q)}) + k (\beta_{j_0}^{(p)}-\beta_{j_0}^{(q)})) k^{2j_0+2}\right| \le \sum_{i=0}^{j_0-1} (k-2)(k+1) k^{2i+2} \le \frac{k-2}{k-1} \cdot k^{2j_0+2}.
    \end{equation*}
    Rearranging the above inequality and using $|\alpha_{j_0}^{(p)} - \alpha_{j_0}^{(q)}| \le k-2$ yields
    \begin{align*}
        8|\beta_{j_0}^{(p)}-\beta_{j_0}^{(q)}| k^{2j_0+2} \le a \le 16 |\beta_{j_0}^{(p)}-\beta_{j_0}^{(q)}| k^{2j_0+3}.
    \end{align*}
    The above inequality together with \eqref{eqn:bc} implies
    \begin{align*}
        \frac{b^2 + c^2}{a} &\ge \frac{(|\beta_{j_0}^{(p)}-\beta_{j_0}^{(q)}| k^{j_0-1})^2}{16 |\beta_{j_0}^{(p)}-\beta_{j_0}^{(q)}| k^{2j_0+3}} \ge \frac{1}{16k^5} \ge 3\delta,\quad \text{and}\\
        \frac{b^2 + c^2}{a} &\le \frac{(2 |\alpha_{i_0}^{(p)}-\alpha_{i_0}^{(q)}| k^{i_0})^2 + (2 |\beta_{j_0}^{(p)}-\beta_{j_0}^{(q)}| k^{j_0})^2}{8|\beta_{j_0}^{(p)}-\beta_{j_0}^{(q)}| k^{2j_0+2}} \le \frac{2 (2 k^{j_0+1})^2}{8|\beta_{j_0}^{(p)}-\beta_{j_0}^{(q)}| k^{2j_0+2}} \le 1 \le 2(1-\delta). 
    \end{align*}
    Finally, applying Lemma~\ref{lem:sarkozy} with $r = b^2+c^2$, we obtain $\intdist{\sqrt{a^2+b^2+c^2}} > \delta$ as required.
\end{proof}

In the following lemma, we formalise the snowflake embedding construction outlined in the introduction.

\begin{lemma}
\label{lem:snowflake_lift}
    Let $A \subset B_k(0, R)$ be a finite set such that $|p - q| \in \mathbb{N}$ for all distinct $p, q \in A$. Suppose that $\phi \,\colon\, (B_k(0,R), |\cdot|^{1/2}) \rightarrow \mathbb{R}^m$ is bilipschitz with constants $0 < c_\phi \le C_\phi$ so that
    \begin{equation*}
        c_\phi |p-q|^{1/2} \le |\phi(p)-\phi(q)| \le C_\phi |p-q|^{1/2}
        \quad \text{for all } p,q\in B_k(0,R),
    \end{equation*}
    and $\phi(0) = 0$. Set
    \begin{equation*}
        \delta_\phi \coloneqq \frac{2c_\phi^2}{3C_\phi^2+2c_\phi^2}.
    \end{equation*}
    If $0 < \delta \le \delta_\phi$ and $\lambda$ satisfies
    \begin{equation*}
        \frac{\sqrt{3\delta}}{c_\phi} \le \lambda \le \frac{\sqrt{2(1-\delta)}}{C_\phi},
    \end{equation*}
    then the set $P_\lambda(A) \coloneqq \set{(p,\lambda\phi(p)) \,\colon\, p \in A} \subset \mathbb{R}^{k+m}$ satisfies \eqref{eqn:main}. Moreover, $P_\lambda(A)$ is contained in the closed Euclidean ball of radius $(R^2 + 2 R)^{1/2}$.
\end{lemma}

\begin{proof}
    It is easy to check that the choice of $\delta_\phi$ ensures that the interval of permissible values of $\lambda$ is nonempty. For distinct $p,q\in A$, write $a = |p-q| \in \mathbb{N}$ and
    \begin{equation*}
        r=\lambda^2|\phi(p)-\phi(q)|^2.
    \end{equation*}
    The bilipschitz bounds give
    \begin{equation*}
        \lambda^2c_\phi^2 a \le r \le \lambda^2C_\phi^2 a,
    \end{equation*}
    and hence, by the choice of $\lambda$, we get
    \begin{equation*}
        3\delta \le \frac{r}{a} \le 2(1-\delta).
    \end{equation*}
    Applying Lemma~\ref{lem:sarkozy} gives
    \begin{equation*}
        \intdist{\left(a^2+r\right)^{1/2}}>\delta.
    \end{equation*}
    The expression on the left-hand side above is exactly the distance between $(p,\lambda\phi(p))$ and $(q,\lambda\phi(q))$. Finally, for $p\in A$, we have $|\phi(p)| \le C_\phi |p|^{1/2} \le C_\phi R^{1/2}$, so
    \begin{equation*}
        |(p,\lambda\phi(p))|^2\le R^2+\lambda^2C_\phi^2R \le R^2 + 2R,
    \end{equation*}
    where the last inequality follows from the choice of $\lambda$.
\end{proof}

\begin{proposition}
\label{prop:lower4}
    There exists $\delta_1>0$ such that for every $0<\delta\le \delta_1$,
    \begin{equation*}
        N_4(X, \delta) = \Omega_\delta(X).
    \end{equation*}
\end{proposition}

\begin{proof}
    By Assouad's sharp embedding theorem for the unit interval equipped with a snowflaked Euclidean metric~\cite{Ass}, there is a bilipschitz embedding
    \begin{equation*}
        \varphi \,\colon\, ([0,1],|\cdot|^{1/2}) \rightarrow \mathbb{R}^3
    \end{equation*}
    with bilipschitz constants $0 < c \le C$. Suppose $X \ge 2$. Let $M = \lfloor X/2\rfloor$ and define the map $\widetilde{\varphi} \,\colon\, [-M,M] \rightarrow \mathbb{R}^3$ by
    \begin{equation*}
        \widetilde{\varphi}(x) = \sqrt{2M} \left[\varphi\left(\frac{x}{2M} + \frac{1}{2}\right) - \varphi\left(\frac{1}{2}\right)\right]\quad \text{ for all } x \in [-M, M].
    \end{equation*}
    It follows that $\widetilde{\varphi}$ is a bilipschitz embedding of $(B_1(0,M), |\cdot|^{1/2})$ into $\mathbb{R}^3$ with the same bilipschitz constants as $\varphi$. Moreover, note that $\widetilde{\varphi}(0) = 0$. Set
    \begin{equation*}
        \delta_1 \coloneqq \frac{2c^2}{3C^2 + 2c^2}.
    \end{equation*}
    Fix $0 < \delta \le \delta_1$ and choose $\lambda$ such that
    \begin{equation*}
        \frac{\sqrt{3\delta}}{c} \le \lambda \le \frac{\sqrt{2(1-\delta)}}{C}.
    \end{equation*}
    Let $A = [M] \subset \mathbb{R}$. This is an integer distance set contained in $B_1(0,M)$. By Lemma~\ref{lem:snowflake_lift}, the set $P_\lambda(A) = \set{(n,\lambda\widetilde{\varphi}(n)) \,\colon\, n \in [M]} \subset \mathbb{R}^4$
    satisfies \eqref{eqn:main}. It remains only to check that these points lie in the ball of radius $X$. Lemma~\ref{lem:snowflake_lift} gives
    \begin{equation*}
        |(n,\lambda\widetilde{\varphi}(n))|^2\le M^2 + 2M \le X^2/4 + X \le X^2
    \end{equation*}
    for all $n \in [M]$. Thus $N_4(X, \delta) \ge M = \lfloor X/2 \rfloor$, which implies the desired result.
\end{proof}

\begin{remark}
    Assouad's construction for the bilipschitz embedding of $([0,1],|\cdot|^{1/2})$ into $\mathbb{R}^3$ may be viewed as a generalised von Koch fractal curve. His sharp embedding theorem also shows that the ambient dimension $3$ is optimal for such an embedding. In particular, the snowflake-lift argument in Proposition~\ref{prop:lower4} genuinely requires an embedding into $\mathbb{R}^3$, so it cannot yield a linear lower bound in dimension $d = 3$.
\end{remark}

\begin{proof}[Proof of Theorem~\ref{thm:lower}]
    Let $\delta_1$ be the constant from Proposition~\ref{prop:lower4}, and set
    \begin{equation*}
        \delta_0 = \min \set{\frac{1}{48\cdot 3^5}, \left(\frac{\ve}{6}\right)^5, \delta_1}.
    \end{equation*}
    If $0 < \delta \le \delta_0$, then Proposition~\ref{prop:lower3} gives
    \begin{equation*}
        N_3(X,\delta) \ge \delta^{6/5}X^{1-6\delta^{1/5}}
        = \Omega_\delta(X^{1-\ve}),
    \end{equation*}
    while Proposition~\ref{prop:lower4} gives $N_4(X,\delta) = \Omega_\delta(X)$ for the same range of $\delta$.
\end{proof}

\section{Concluding remarks}
\label{sec:conclusion}

Owing to the linear lower bound in dimension four (see Theorem~\ref{thm:lower}), the remaining central question from Erd\H{o}s and S\'ark\"ozy is whether $N_d(X, \delta)$ can be superlinear in $X$ for some fixed dimension $d$ and some fixed $\delta > 0$. It is striking that their question isolates exactly the threshold that becomes natural from the snowflake-lift viewpoint, even though the linear construction in dimension four comes from a tool, Assouad's theorem, that is quite different from the original planar constructions.

This viewpoint suggests a closely related extremal problem for integer distance sets. For $k \in \mathbb{N}$ and $X > 0$, let $I_k(X)$ denote the largest size of a set $A\subset B_k(0,X)$ such that $|p-q| \in \mathbb{N}$ for all distinct $p, q \in A$. The trivial construction on a line gives $I_k(X) \ge \lfloor 2X \rfloor + 1$ for every $k \ge 1$ and $X > 0$. For $k = 1$, this trivial construction is clearly optimal. For $k = 2$, the linear construction is optimal up to constants~\cite[Theorem~2]{Sol}. For $k \ge 3$, a general upper bound follows from the distinct distances theorem of Solymosi and Vu~\cite{SV}. They prove that any $n$-point subset of $\mathbb{R}^k$ determines at least $\Omega_k(n^{2/k - 2/(k(k+2))})$ distinct distances. Since an integer distance subset of $B_k(0,X)$ determines at most $2X$ distinct distances, this yields
\begin{equation*}
    I_k(X)=O_k\left(X^{\frac{k+1}{2} - \frac{1}{2(k+1)}}\right).
\end{equation*}
Thus there is a substantial gap between the known lower and upper bounds for $I_k(X)$.

\begin{problem}
    For $k \ge 3$, determine the order of growth of $I_k(X)$. Does there exist an integer $k \ge 3$ such that $\lim_{X \rightarrow \infty} I_k(X)/X = \infty$?
\end{problem}

Lemma~\ref{lem:snowflake_lift} and the proof of Proposition~\ref{prop:lower4} show that whenever $([0, 1]^k,\|\cdot\|_2^{1/2})$ admits a bilipschitz embedding into some Euclidean space $\mathbb{R}^m$, lower bounds for $I_k(X)$ lift to lower bounds for $N_{k+m}(X, \delta)$, after changing the radius by only a constant factor depending on $\delta$ and the embedding. Therefore, any superlinear lower bound for $I_k(X)$, combined with a snowflake embedding, would give a corresponding superlinear lower bound for $N_d(X, \delta)$ in some higher dimension. However, we suspect that the true growth of $I_k(X)$ is indeed linear in $X$. The methods of \cite{GIP} could be useful in showing this. We remark that the bounds for the distinct distances problem for typical norms on $\mathbb{R}^d$~\cite{ABS} imply a linear upper bound on the size of integer distance sets. On the other hand, integer distance sets in the $\ell^1$ and $\ell^\infty$ norms have size $\Theta_k(X^k)$.

It may also be useful to distinguish between lattice and non-lattice constructions. The three-dimensional construction in Section~\ref{sec:lower}, like S\'ark\"ozy's planar construction~\cite{Sar2}, is a subset of the integer lattice. In contrast, the snowflake-lift construction in Proposition~\ref{prop:lower4} is generally not a lattice set. It would be interesting to know whether the upper bounds for $N_d(X, \delta)$ can be improved under a lattice restriction; techniques from additive combinatorics may be relevant in this regard since the difference sets of such lattice sets are restricted.

One can also ask for analogues for other norms. For $X > 0$, $\delta \in (0, 1/2)$, and a norm $\mathcal{N}$ on $\mathbb{R}^d$, let $N_{\mathcal{N}}(X, \delta)$ denote the maximum number of points that can be chosen from a closed ball of radius $X$ such that all pairwise distances are at least $\delta$ away from any integer. A lower bound of $\Omega_\delta(1)$ is trivial, while an upper bound of $O_\delta(X^{d-1})$ follows from a slabs argument analogous to Lemma~\ref{lem:ind}. For $\ell^p$-norms, the cases $p = 1$ and $p = \infty$ admit an $O_\delta(1)$ upper bound by an argument similar to Lemma~\ref{lem:oned}. For $1 < p < \infty$, one can use the following analogue of Lemma~\ref{lem:sarkozy} to obtain constructions analogous to those in Section~\ref{sec:lower}.

\begin{lemma}
    Let $\delta \in (0, 1/2)$, $p > 1$, and $a \in \mathbb{N}$. Suppose $r$ is a positive real number satisfying
    \begin{equation*}
        p \delta (3/2)^{p-1} \le \frac{r}{a^{p-1}} \le p(1 - \delta).
    \end{equation*}
    Then \begin{equation*}
        \intdist{(a^p + r)^{1/p}} > \delta.
    \end{equation*}
\end{lemma}

In particular, generalising S\'ark\"ozy's construction to $\ell^p$-norms, for every $\ve > 0$, one obtains
\begin{equation*}
    N_{\|\cdot\|_p^{(\lceil p/(p-1) \rceil + 1)}}(X, \delta) = \Omega_\delta(X^{1-\ve})
\end{equation*}
for $\delta$ sufficiently small in terms of $\ve$, where $\|\cdot\|_p^{(d)}$ denotes the $\ell^p$-norm in $\mathbb{R}^d$. However, as in the Euclidean case, this particular construction does not extend to dimensions $d > \lceil p/(p-1) \rceil + 1$. Further, generalising the snowflake-lift construction to $\ell^p$-norms yields
\begin{equation*}
    N_{\|\cdot\|_p^{(\lfloor p/(p-1) \rfloor + 2)}}(X, \delta) = \Omega_\delta(X).
\end{equation*}
This suggests an analogue of the Erd\H{o}s--S\'ark\"ozy superlinear question for general norms.

\begin{problem}
    Does there exist an integer $d \in \mathbb{N}$, a constant $\delta \in (0, 1/2)$, and norm $\mathcal{N}$ on $\mathbb{R}^d$ such that $\lim_{X \rightarrow \infty} N_{\mathcal{N}}(X, \delta)/X = \infty$?
\end{problem}

Finally, let us mention a related problem first considered by Erd\H{o}s, Graham, and S\'ark\"ozy (see~\cite{Erd2}). For $X > 0$, what is the maximum measure $f(X)$ of a measurable set $A \subseteq B_2(0, X)$ with no integer distances, that is, such that $|x-y| \not \in \mathbb{N}$ for any distinct $x, y \in A$? The obvious upper bound of $O(X)$ is obtained by considering integer translates of the set along a fixed direction. A lower bound of $\Omega(X^{1/2-o(1)})$ follows by taking S\'ark\"ozy's construction from~\cite{Sar2} and placing an open ball of radius $\delta/2$ at each point, with $\delta \rightarrow 0$. Obtaining a sufficiently strong dependence on $\delta$ in Konyagin's bound $N_2(X,\delta) = O_\delta(X^{1/2})$ would imply sharper upper bounds for measurable sets that are unions of balls. It would also be interesting to understand the corresponding measurable problem in higher dimensions. We also mention here the work of Kurz and Mishkin~\cite{KM} who determine the maximum measure of open sets with $n$ connected components and no integer distances.

\bibliographystyle{amsplain}
\bibliography{bibliography}

\end{document}